\documentclass[a4paper,12pt]{scrartcl}

\usepackage[utf8]{inputenc}
\usepackage{amsfonts,amssymb}
\usepackage{amsmath}
\usepackage{graphicx}
\usepackage{color}

\usepackage{a4wide}

\usepackage{empheq}
\numberwithin{equation}{section}
\usepackage{amsthm}
\theoremstyle{plain}
\newtheorem{thm}{Theorem}[section]
\newtheorem*{mainthm}{Main Theorem}

\newtheorem{lem}[thm]{Lemma}
\newtheorem{cor}[thm]{Corollary}
\newtheorem{prop}[thm]{Proposition}
\newtheorem{obs}[thm]{Observation}

\theoremstyle{definition}

\newtheorem{exmpl}[thm]{Example}

\theoremstyle{remark}

\newcommand{\N}{\mathbb{N}}
\newcommand{\Z}{\mathbb{Z}}

\usepackage{ifthen}

\newcommand{\defeq}{\mathrel{\mathop{:}}=}

\newcommand{\gen}[1]{\left\langle #1\right\rangle}
\newcommand{\abs}[1]{\left\lvert#1\right\rvert}
\newcommand{\realize}[1]{\left\lvert#1\right\rvert}

\newcommand{\CAT}{\mathrm{CAT}}

\newcommand{\lab}{c}
\newcommand{\len}{\ell}
\newcommand{\Graph}{\Gamma}
\newcommand{\ExtVert}{\hat{V}}
\newcommand{\ExtGraph}{\Delta}
\newcommand{\ExtFlag}{L}
\newcommand{\GProd}[1]{G(#1)}
\newcommand{\qgen}[1]{[#1]}
\newcommand{\fin}{\mathrm{fin}}
\renewcommand{\inf}{\mathrm{inf}}
\newcommand{\Cliques}{\mathcal{S}}
\newcommand{\CliqueCosets}{\mathcal{T}}
\newcommand{\Weirds}{\mathcal{P}}
\newcommand{\WeirdCosets}{\mathcal{Q}}
\newcommand{\Chamber}{K}
\newcommand{\Complex}{X}
\newcommand{\WeirdComplex}{Y}
\newcommand{\lk}{\operatorname{lk}}
\newcommand{\deslk}{\lk^{\downarrow}}
\newcommand{\ulk}{\operatorname{ulk}}
\newcommand{\dlk}{\operatorname{dlk}}
\newcommand{\desulk}{\ulk^{\downarrow}}
\newcommand{\desdlk}{\dlk^{\downarrow}}

\title{$\mathbf{CAT(0)}$ cubical complexes\\for graph products\\of finitely generated abelian groups}
\author{Kim Ruane and Stefan Witzel}
\date{November 2013}

\begin{document}

\maketitle

\begin{abstract}
\noindent{}\textsc{Abstract.} We show that every finite graph product of finitely generated abelian groups acts properly and cocompactly on a $\CAT(0)$ cubical complex. The complex generalizes (up to subdivision) the Salvetti complex of a right-angled Artin group and the Coxeter complex of a right-angled Coxeter group. In the right-angled Artin group case it is related to the embedding into a right-angled Coxeter group described by Davis and Januszkiewicz. We compare the approaches and also adapt the argument that the action extends to finite index supergroup that is a graph product of finite groups.
\end{abstract}

Let $\Graph$ be a (simplicial) graph in which every vertex is labeled by a finitely generated abelian group. The graph product $\GProd\Graph$ is the free product of all these vertex groups modulo the relations that the elements of two of them commute if they are connected by an edge. The purpose of this article is to show:

\begin{mainthm}
Let $\Gamma$ be a finite graph with vertices labeled by finitely generated abelian groups. There is a $\CAT(0)$ cubical complex on which $\GProd\Graph$ acts properly and cocompactly.
\end{mainthm}

Every graph product of finitely generated abelian groups can also be written as a graph product of cyclic groups and we will restrict to that case.  If all of the vertex groups are infinite cyclic, then the resulting graph product is a right-angled Artin group.  If all vertex groups are $\Z/2\Z$, then the resulting graph product is a right-angled Coxeter group.  For both of these classes, there is a standard construction of a locally finite $\CAT(0)$ cube complex on which the graph product acts geometrically (properly and cocompactly), see for example \cite[2.5, 2.6]{charney07}. For general cyclic vertex groups Davis \cite[Section~5]{davis98} constructs a right-angled building on which the graph product acts, extending the construction for right-angled Coxeter groups.  However, this complex is locally finite (and the action is proper) only if all of the vertex groups are finite. Our construction treats the infinite cyclic groups as in the right-angled Artin case. In fact, it would work in the same way for graph products of finite groups and infinite cyclic groups at the expense of making the notation more cumbersome.

A different approach for right-angled Artin groups is to embed every infinite cyclic group into an infinite dihedral group. This was used by Davis and Januszkiewicz \cite{davjan00} to show that right-angled Artin groups embed as finite index subgroups into right-angled Coxeter groups. As we will see our construction is closely related to theirs. The approach is from the opposite direction, however: we first verify the group theoretic connection and deduce the relation between the spaces from it.

Our construction also gives an alternative proof that a graph product of finitely generated abelian groups is virtually cocompact special. This was originally shown by Kim \cite[Theorem~3(2)]{kim12}.

\medskip

The paper is organized as follows. In Section~\ref{sec:basics} we recall some facts about graph products of groups and set notation. The cubical complex is constructed in Section~\ref{sec:complex} and shown to be contractible and $\CAT(0)$ in Section~\ref{sec:contractibility}. In Section~\ref{sec:comparison} we make the connection to the Davis--Januszkiewicz construction.

\medskip

\textbf{Acknowledgements.} The work for this article was started while the second author was visiting the first author. The second author is grateful for the hospitality of Tufts University and for financial support through the DAAD which partially funded the visit. He also gratefully acknowledges support through the SFB 878 in Münster. The authors are grateful to S.\ Kim for pointing out the connection with his work on virtually cocompact special groups.

\section{Graph products of cyclic groups}
\label{sec:basics}

As mentioned in the introduction we may restrict ourselves to graph products of cyclic groups and we introduce notation accordingly. Let $\Graph$ be a simplicial graph with vertex set $V$ and edge set $E$. Let $\lab \colon V \to \N \cup \{\infty\}$ be a vertex-labeling. We define the graph product of $\Graph$ to be the group $\GProd\Graph$ with presentation
\[
\GProd\Graph \defeq \gen{s \in V \mid s^{\lab(s)} = 1 \text{ for } s \in V, [s,t]=1 \text{ for } \{s,t\} \in E} \text{ .}
\]
Thus every vertex $s \in V$ is an element of order $\lab(s)$ in $\GProd\Graph$.

We will need a fact about general graph products which provides a solution to the word problem (with the elements of the vertex groups as generators). It was first proved by Green \cite{green90} and later reproved using geometric methods by Hsu and Wise \cite{hsuwis99}. We only describe it in our case of cyclic vertex groups.

Every element $g \in \GProd\Graph$ can be written as a word $(s_1^{e_1},\ldots,s_k^{e_k})$ with each $s_i$ one of the vertices of $\Graph$ (not necessarily distinct) and $e_i \in \Z$. By that we mean that the product $s_1^{e_1} \cdots s_k^{e_k}$ in $\GProd\Graph$ equals $g$. The following operations on words clearly do not change the element of $\GProd\Graph$ that the word describes:
\begin{enumerate}
\item Remove the letter $1(=s_i^{0})$.\label{item:throw_away}
\item Replace two consecutive letters $s_i^{e_i}$ and $s_{i+1}^{e_{i+1}}$ which are powers of the same generator $s_i=s_{i+1}$ by $s_i^{e_i+e_{i+1}}$.\label{item:merge}
\item Replace two consecutive letters $s_i^{e_i}, s_{i+1}^{e_{i+1}}$ such that $s_i$ and $s_{i+1}$ are connected by an edge in $\Gamma$ by  $s_{i+1}^{e_{i+1}},s_i^{e_i}$.\label{item:interchange}
\end{enumerate}
A word that cannot be shortened using these operations is called \emph{reduced}. The following is \cite[Theorem~2.5]{hsuwis99}.

\begin{thm}
A reduced word describes the identity element if and only if it is empty.
\end{thm}

By induction on the word length one obtains the seemingly stronger version formulated by Green \cite[Theorem~3.9]{green90}.

\begin{cor}
\label{cor:green_strong}
Two words describe the same element if and only if they can be transformed into a common word using only the operations \eqref{item:throw_away} to \eqref{item:interchange}.
\end{cor}

In particular a word has minimal length (among those describing the corresponding element) if and only if it is reduced. For our purposes, a slightly different measurement of length will be useful. It corresponds to the generating set of $\GProd\Graph$ consisting of all the elements of the finite vertex groups, but only one generator for each infinite cyclic subgroup.

So if we partition the vertex set of $\Graph$ into
\[
V_\fin=\{s \in V \mid \lab(s) < \infty\} \quad \text{and} \quad V_\inf=\{s \in V \mid \lab(s) = \infty\}
\]
the length of a word is given by
\[
\len(s_1^{e_1},\ldots,s_k^{e_k}) = \sum_{s_i \in V_\inf} \abs{e_i} + \sum_{s_i \in V_\fin} 1 \text{ .}
\]
It follows from Corollary~\ref{cor:green_strong} that every reduced word has minimal length with respect to this length function though the converse is not true. The length of an element of $\GProd\Graph$ is defined to be the minimal length of a word representing it (for example a reduced word).

We say that an element $g \in \GProd\Graph$ \emph{ends with} $s \in V$ if there is a reduced word $(s_1^{e_1},\ldots,s_k^{e_k})$ describing $g$ with $s = s_k$. Another consequence of Corollary~\ref{cor:green_strong} is:

\begin{cor}
\label{cor:endletters_commute}
If $g$ ends with $s$ and ends with $t$ then $s$ and $t$ commute (that is, are connected by an edge).
\end{cor}

\section{The complex}
\label{sec:complex}

From now on fix a finite graph $\Graph$ with labeling $\lab$ and let $G \defeq \GProd\Graph$ be the associated group. Let $V^\pm = V \cup V^{-1}$ denote the set of generators and their inverses of $\GProd\Graph$ and define $V_\inf^\pm$ accordingly. We define a new graph $\ExtGraph$ whose vertex set is $\ExtVert \defeq V_\fin \cup V_\inf^\pm$ and whose edges are given by pulling back the edges of $\Graph$ via the obvious projection $V \to \ExtVert$ (note that this means that $s$ and $s^{-1}$ are \emph{not} connected for $s \in V_\inf$).

For $s \in V^\pm$ we define the expression
\[
\qgen{s} \defeq \left\{
\begin{array}{ll}
\gen{s}\text{ ,} & s \in V_\fin^\pm\medskip\\
\{1,s\}\text{ ,} & s \in V_\inf^\pm \text{ .}
\end{array}
\right.
\]
We extend this expression to cliques (complete subgraphs) of $\ExtGraph$ by setting $\qgen{C} = \qgen{s_1}\cdots\qgen{s_k}$ (element-wise product) if $s_1,\ldots,s_k$ are the vertices of $C$. Note that the order of the product does not matter since all the $s_i$ commute. Moreover the clique $C$ can be recovered from $\qgen{C}$, its vertices being just $\qgen{C} \cap \ExtVert$.

Let $\Cliques$ be the poset of the $\qgen{C}$ where $C$ ranges over cliques of $\ExtGraph$ and let $\Chamber\defeq \realize{\Cliques}$ be its realization. That is, the vertices of $\Chamber$ are the sets $\qgen{C}$ and the simplices are flags of those, ordered by inclusion. The poset $\Cliques$ is covered by the intervals $[\qgen{\emptyset},\qgen{C}]$ which are boolean lattices. Therefore $\Chamber$ naturally carries a cubical structure in which intervals are cubes (see \cite[Proposition~A.38]{abrbro08}). The link of $\qgen{\emptyset}$ in this cubical structure is the flag complex of $\ExtGraph$.

Let $\CliqueCosets \defeq G\Cliques$ be the set of cosets of elements of $\Cliques$. This set is again ordered by inclusion and its realization $\Complex = \Complex(\Gamma) \defeq \realize{\CliqueCosets}$ is the space we are looking for. For the same reason as before $\Complex$ can be regarded as a cubical complex. Note that if $\Graph$ has no vertices labeled $\infty$, then $\Graph = \ExtGraph$ and $\Complex(\Graph)$ is just the usual coset-complex (the right-angled building from \cite[Section~5]{davis98}).

\begin{obs}
The subspace $\Chamber$ is a weak fundamental domain for the action of $G$ on $\Complex$.
\end{obs}

Note that $\Complex$ can equivalently be described as $G \times \Chamber /{\sim}$ where $\sim$ is the equivalence relation generated by $(g,x) \sim (gt,x)$ if $t \in \gen{s}$ and $x \in [\gen{t},\gen{C}]$ for $t \in C \subseteq V_\fin$ and $(gs,x) \sim (gs^{-1},y)$ for $s \in V_\inf$ if the barycentric coordinates of $y$ can be obtained from those of $x$ by replacing $s^{-1}$ by $s$.

\begin{obs}
The complex $\Complex$ is locally finite and the action of $G$ is proper.
\end{obs}

\begin{proof}
This essentially follows from the finiteness of the sets $\qgen{C}$.
\end{proof}

An important class of cubical complexes are the \emph{special} ones introduced by Haglund and Wise \cite{hagwis08}. There is also a notion of when the action of a group $G$ on a cubical complex $\Complex$ is special \cite[Definition~3.4]{hagwis10}, which implies that specialness of $\Complex$ is inherited by $G \backslash \Complex$ (and in fact by $H \backslash \Complex$ for every $H \le G$) \cite[Theorem~3.5]{hagwis10}. In view of the known cases for right-angled Artin and Coxeter groups \cite[Example~3.3~(ii)]{hagwis08} it is not surprising that our construction satisfies the necessary conditions:

\begin{obs}
\label{obs:special_action}
The action of $G$ on $\Complex$ is special.
\end{obs}

\begin{proof}
Every edge of $\Complex$ is of the form $gs^k\qgen{C} \le g s^k\qgen{C \cup \{s\}}$ with $k \in \Z$, $g$ not ending with $s$, and $s \in \ExtVert \setminus C$. So every edge is naturally oriented and labeled by some $s^k$. Orientation and label are preserved along walls and under the action of $G$. Edges with same label cannot cross and cannot osculate. Edges with label $s^k$ and $t^\ell$ can cross only if $s$ and $t$ commute in which case no two such edges can osculate. The conditions in \cite[Definition~3.4]{hagwis10} now follow.
\end{proof}

\section{Contractibility and $\mathbf{CAT(0)}$-ness}
\label{sec:contractibility}

One reason for the popularity of cubical complexes is the combination of the following theorems, see \cite[Theorem~II.5.20]{brihae} and \cite[Theorem~II.4.1]{brihae}.

\begin{thm}
\label{thm:gromov}
A finite dimensional cubical complex has non-positive curvature if and only if the link of each of its vertices is a flag complex.
\end{thm}

\begin{thm}
\label{thm:cartan-hadamard}
A complete connected metric space that is non-positively curved and simply connected is $\CAT(0)$.
\end{thm}

Using both results, we only have to show that $\Complex$ is simply connected and has flag complexes as links to conclude that $\Complex$ is $\CAT(0)$. However, in our case showing that $\Complex$ is simply connected is not significantly easier than showing it to be contractible. We will therefore directly show:

\begin{thm}
\label{thm:contractible}
$\Complex$ is contractible.
\end{thm}

We will prove Theorem~\ref{thm:contractible} by building up from a vertex to the whole complex while taking care that contractibility is preserved at each step. The technical tool to do this is combinatorial Morse theory.

A \emph{Morse function} on an affine cell complex $\Complex$ is a map $f \colon \Complex^{(0)} \to \Z$ such that every cell $\sigma$ of $\Complex$ has a unique vertex in which $f$ attains its maximum. The \emph{descending link} $\deslk v$ of a vertex $v$ consists of those cofaces $\sigma$ for which $v$ is that vertex. We denote the sublevel set $f^{-1}((-\infty,n])$ by $X^{\le n}$. The Morse lemma in its most basic form --- which is good enough for us --- can be stated as follows.

\begin{lem}
\label{lem:morse_lemma}
If $X^{\le n-1}$ is contractible and for every vertex $v$ with $f(v)$ the descending link $\deslk v$ is contractible then $X^{\le n}$ is contractible.
\end{lem}

Returning to our concrete setting, we first study the links of vertices in $\Complex$. Let $\ExtFlag$ denote the flag complex of $\ExtGraph$.

\begin{obs}
The link of a vertex of the form $g\qgen{\emptyset} = \{g\}$ is isomorphic to $\ExtFlag$. Indeed, the correspondence $C \mapsto g\qgen{C}$ is a bijection between the faces of $\ExtFlag$ and the cofaces of $g\qgen{\emptyset}$.
\end{obs}

The link of a general vertex is not much more complicated:

\begin{obs}
\label{obs:up-down-link}
Let $g\qgen{C}$ be a vertex of $\Complex$. The link decomposes as
\[
\lk g\qgen{C} = \ulk g\qgen{C} * \dlk g\qgen{C}
\]
into an up-link $\ulk g\qgen{C}$ with simplices $g\qgen{C'}$ with $C' \supsetneq C$, and a down-link $\dlk g\qgen{C}$ with simplices $gh \qgen{C'}$ with $C' \subsetneq C$, $h \in \qgen{C \setminus C'}$. The up-link is isomorphic to the link of $C$ in $\ExtFlag$. The down-link is isomorphic to the join of the sets $\qgen{s}, s \in C$.
\end{obs}

\begin{figure}[thb]
\begin{center}
\includegraphics{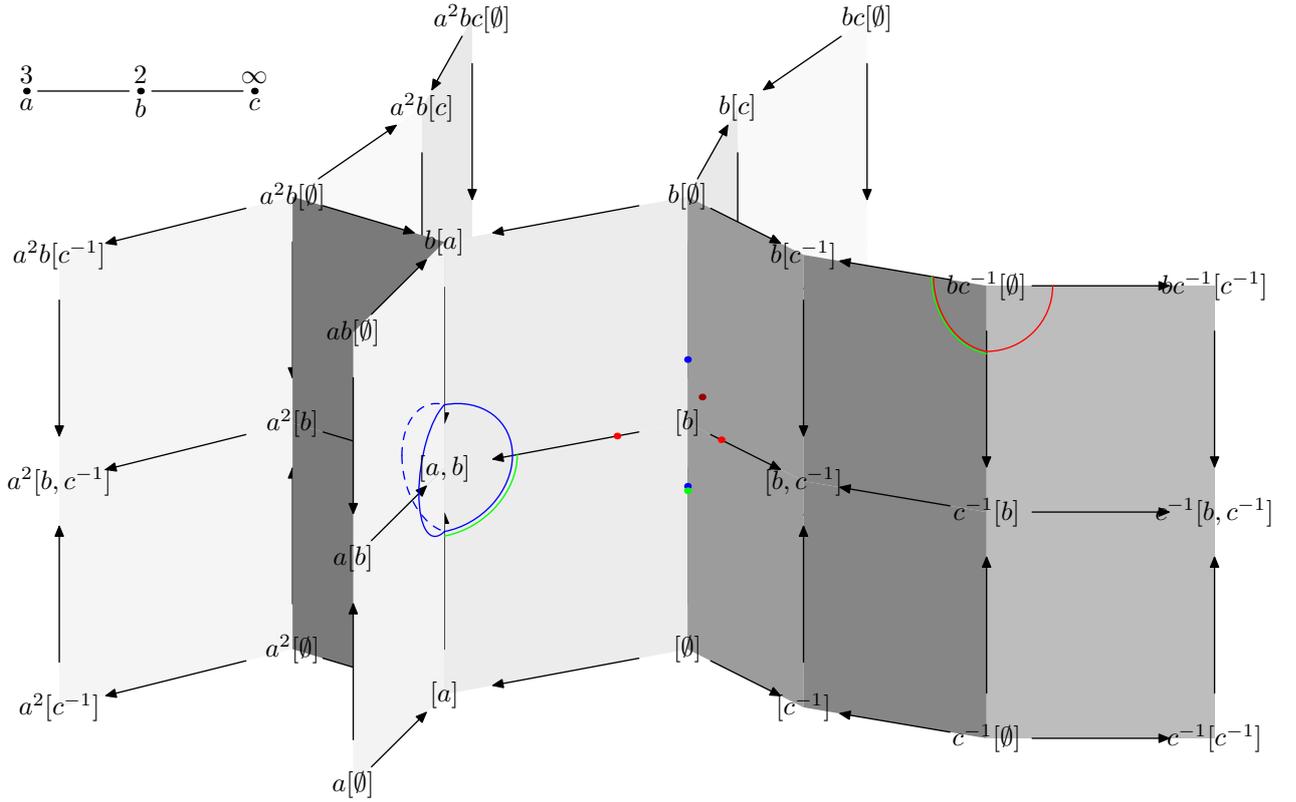}
\caption{Part of the complex $\Complex(\Gamma)$ for the graph $\Gamma$ indicated in the upper left corner. Some partial links are drawn. Down-links are blue, up-links are red. Descending links are green. The arrows indicate the poset relation. The ``bottom layer'' of the complex is the complex of the graph obtained from $\Gamma$ by deleting the vertex $b$, see Figure~\ref{fig:example_1d}.}
\label{fig:example_2d}
\end{center}
\end{figure}

Note that the \emph{simplices} in the link of a vertex correspond to \emph{vertices} in the ambient complex. This is a feature of cubical complexes, see Figure~\ref{fig:example_2d}. Observation~\ref{obs:up-down-link} implies in particular:

\begin{cor}
\label{cor:flag}
The link of every vertex of $\Complex$ is a flag complex.
\end{cor}

\begin{figure}[thb]
\begin{center}
\includegraphics{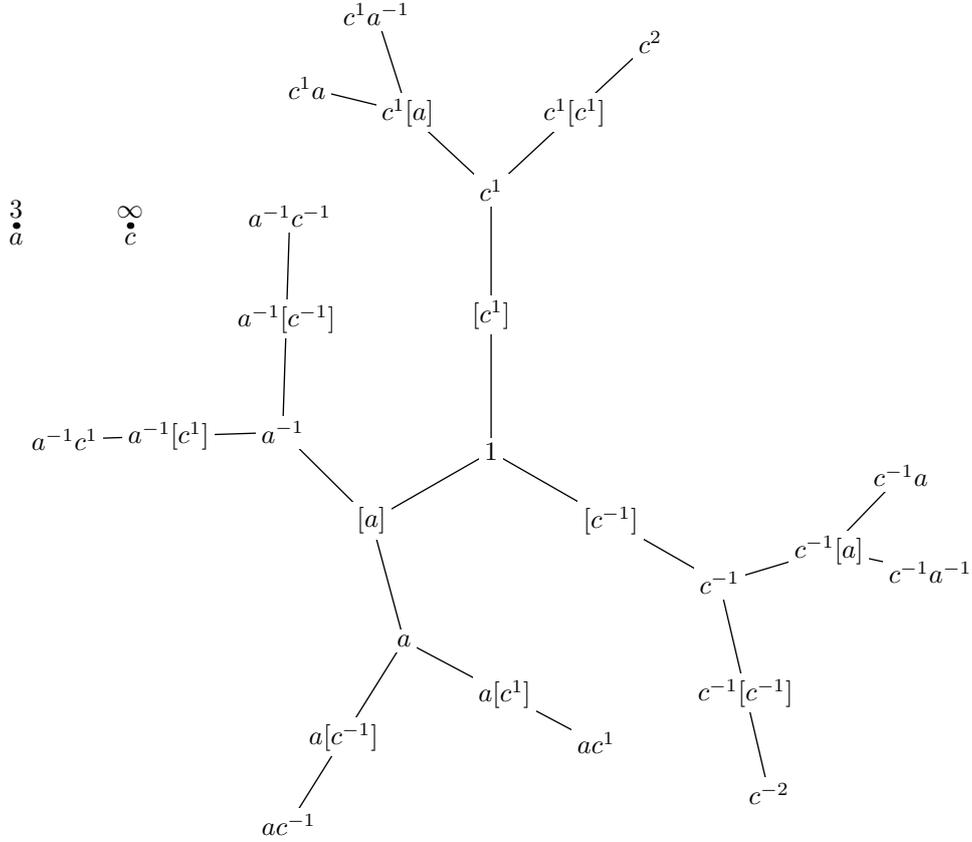}
\caption{Part of the complex $\Complex(\Gamma)$ for the graph $\Gamma$ indicated in the upper left corner.}
\label{fig:example_1d}
\end{center}
\end{figure}

We are now ready to start the proof of Theorem~\ref{thm:contractible}. The Morse function we will be using is
%\vspace{-\medskipamount}
\begin{align*}
f \colon \Complex &\to \Z \times \Z\\
g\qgen{C} &\mapsto (\max \len(g\qgen{C}),-\#C)
\end{align*}
where $\len$ is length of an element as defined in Section~\ref{sec:basics} and the codomain is ordered lexicographically. This is formally not a Morse function as defined before since the codomain is not $\Z$. However, since $\#C$ is uniformly bounded throughout $\Complex$, the image of $f$ is order-isomorphic to $\Z$. The actual condition for being a Morse function is satisfied:

\begin{obs}
The function $f$ attains its maximum over the vertices of a cube of $\Complex$ in a unique vertex.
\end{obs}

\begin{proof}
Let $g[\qgen{C_1},\qgen{C_2}]$ be an arbitrary cube. Let $gh \in g\qgen{C_2}$ have maximal length. There is a unique such element because all the elements of $\qgen{C_2}$ commute. Let $C'$ be the set of all $s \in C_2$ that $gh$ ends with. Then $f$ attains its maximum over $g[\qgen{C_1},\qgen{C_2}]$ in $gh\qgen{C_1 \cup C'}$. Indeed, enlarging it drops the secondary height and reducing it is impossible while staying above $g\qgen{C_1}$ and $gh$.
\end{proof}

To apply Lemma~\ref{lem:morse_lemma} we have to show that descending links are contractible. It follows from Observation~\ref{obs:up-down-link} that the descending link of a vertex decomposes as $\deslk v = \desulk v * \desdlk v$ into a descending up-link $\desulk v = \deslk v \cap \ulk v$ and a descending down-link $\desdlk v = \deslk v \cap \dlk v$. Since the join of a contractible complex with any complex is contractible, it suffices to show that one of the two join factors is contractible. We treat vertices of the form $g\qgen{\emptyset}$ and the other ones separately.

\begin{lem}
\label{lem:down-link}
Let $g\qgen{C}$ be a vertex of $\Complex$. If $C \ne \emptyset$ then the descending down-link is contractible.
\end{lem}

\begin{proof}
Let $s \in C$. The vertices corresponding to the join factor $\qgen{s}$ in Observation~\ref{obs:up-down-link}
point to the vertices $gh \qgen{C \setminus \{s\}}, h \in \qgen{s}$. Such a vertex is ascending unless $\max \len(gh \qgen{C \setminus \{s\}}) < \max \len(g \qgen{C}$). This is the case precisely for the element $h \in \qgen{s}$ such that $g$ ends with $h^{-1}$ if such an element exists and is the case for $h = 1$ otherwise. In either case it is a single point. Thus the descending down-link is a join of singleton sets, that is, a simplex.
\end{proof}

\begin{lem}
\label{lem:up-link}
The descending up-link of a vertex of the form $g\qgen{\emptyset}$ is contractible provided $g \ne 1$.
\end{lem}

\begin{proof}
A vertex of the form $g\qgen{C}$ is descending for $g\qgen{\emptyset}$ unless $\max \len(g\qgen{C}) > \len(g)$, that is, unless $C$ contains a letter that $g$ does not end with. In other words $g\qgen{C}$ is descending if $C$ contains only letters that $g$ ends with. The poset of such $g\qgen{C}$ is a barycentrically subdivided simplex by Corollary~\ref{cor:endletters_commute} (unless $g = 1$ in which case it is empty).
\end{proof}

\begin{proof}[Proof of Theorem~\ref{thm:contractible}]
The sublevel set $\Complex^{\le (0,0)}$ contains vertices $g\qgen{C}$ whose longest elements have length $0$. The only possibility for this is $\{1\}$, so $\Complex^{\le (0,0)}$ consists of a single vertex and in particular is contractible. Using that the descending link of every vertex is contractible by Lemmas~\ref{lem:down-link} and \ref{lem:up-link}, an inductive application of Lemma~\ref{lem:morse_lemma} shows that every sublevel set of $\Complex$ is contractible. But the whole complex is the limit of these, thus contractible as well.
\end{proof}

From Theorems~\ref{thm:gromov}, \ref{thm:cartan-hadamard} and \ref{thm:contractible} and Corollary~\ref{cor:flag} we get:

\begin{cor}
\label{cor:cat0}
$\Complex$ is $\CAT(0)$.
\end{cor}

Together with Observation~\ref{obs:special_action} we can also conclude:

\begin{cor}
$G$ is virtually cocompact special.
\end{cor}

\begin{proof}
By Corollary~\ref{cor:cat0} $\Complex$ is $\CAT(0)$ and hence special (see \cite[Example~3.3~(ii)]{hagwis08}).
Every torsion element in $G$ fixes a cell of $\Complex$ \cite[Corollary~II.2.8(1)]{brihae} and hence is conjugate to an element of $\gen{C}$ for some clique $C \subseteq V_\fin$. Thus any nontrivial torsion element is mapped to a nontrivial element under the projection $G \to \prod_{s \in V_\fin} \gen{s}$. So the kernel $K$ of this map acts freely and cocompactly on $\Complex$. The quotient $K \backslash \Complex$ is special by Observation~\ref{obs:special_action}.
\end{proof}

\section{Comparison to Davis--Januszkiewicz}
\label{sec:comparison}

When all vertices of $\Gamma$ are labeled $\infty$, then $\GProd\Gamma$ is a right-angled Artin group. For that case Davis and Januszkiewicz constructed graphs $\Gamma'$ and $\Gamma''$, which in our notation would have every vertex labeled $2$. They showed that $\GProd\Gamma$ embeds as a finite index subgroup into $\GProd{\Gamma''}$ which in turn acts on the Coxeter complex of $\GProd{\Gamma'}$. We want to explain how their construction carries over to graph products of general cyclic groups and how it relates to the construction from Section~\ref{sec:complex}.

The graph $\Gamma'$ has vertices $V_\fin \cup (V_\inf \times \{-1,1\})$ and edges are given by pulling back the edges from $\Gamma$ via the obvious projection to $V$. The graph $\Gamma''$ has vertices $V_\fin \cup (V_\inf \times \{0,1\})$ and the following edges: the subgraph on $V_\fin \cup (V_\inf \times \{1\})$ is canonically isomorphic to $\Gamma$; and a vertex $(v,0)$ is connected to every other vertex except for $(v,1)$. Each of the vertices in $V_\fin$ keeps its label and the vertices in $(V_\inf \times \{-1,0,1\})$ are labeled by $2$. For a vertex $i \in V_\inf$ Davis--Januzskiewicz denote these elements $g_i = i$, $s_i = (i,1)$, $t_i=(i,-1)$, and $r_i=(i,0)$ respectively.

The graph products associated to these groups are related via the maps
\begin{align*}
\beta \colon \GProd\Gamma &\to \GProd{\Gamma''}\\
s &\mapsto s\text{ ,} \quad s \in V_\fin\\
s &\mapsto (s,1) \cdot (s,0)\text{ ,}\quad s \in V_\inf
\end{align*}
and
\begin{align*}
\alpha \colon \GProd{\Gamma'} &\to \GProd{\Gamma''}\\
s &\mapsto s\text{ ,} \quad s \in V_\fin\\
(s,1) &\mapsto (s,1)\text{ ,} \quad s \in V_\inf\\
(s,-1) & \mapsto (s,0) \cdot (s,1) \cdot (s,0)\text{ ,} \quad s \in V_\inf
\end{align*}
which are easily seen to be injective. In fact, letting $E$ denote the subgroup of $\GProd{\Graph''}$ generated by the elements $(s,0), s \in V_\inf$ we see that $\GProd{\Graph''}$ can be written as semidirect products
\begin{equation}
\label{eq:semidirect_product}
\GProd\Graph \rtimes E = \GProd{\Graph''} = \GProd{\Graph'} \rtimes E
\end{equation}
where the action is always trivial on $V_\fin$ and on the remaining generators is given by
\[
s^{(t,0)} = \left\{
\begin{array}{ll}
s&s \ne t\\
s^{-1}& s = t
\end{array}
\right.
\quad
\text{respectively}\quad
(s,\pm 1)^{(t,0)} =
\left\{
\begin{array}{ll}
(s,\pm1)&s \ne t\\
(s,\mp 1)& s = t\text{ .}
\end{array}
\right.
\]
This can be seen for example by writing
\begin{align*}
\GProd{\Gamma''} = \langle V \cup (V_\inf \times \{-1,0,1\}) \mid{} &\text{all the previous relations},\\
&s=(s,1) \cdot (s,0),\text{ for } s \in V_\inf,\\
&(s,-1) = (s,0)\cdot (s,1) \cdot (s,0)\text{ for } s \in V_\inf\rangle
\end{align*}
and then applying Tietze transformations to remove generators. These algebraic considerations play the role of the geometric arguments in \cite{davjan00}. One of the main ingredients here is that the groups $\GProd\Graph$ and $\GProd{\Graph'}$ do indeed admit the described actions of the group $E$. The basic example to keep in mind is the following.

\begin{exmpl}
If $\Graph$ has just one vertex labeled $\infty$, then $X(\Graph)$ can be thought of as the real line. Then $\GProd{\Graph}$ is the group generated by an element $s$ which is translation by $2$. The group $\GProd{\Graph'}$ is generated by elements $(s,1)$ and $(s,-1)$ which are reflection at $1$ and $-1$ respectively. Finally, $\GProd{\Graph''}$ is generated by elements $(s,1)$ and $(s,0)$ which are reflection at $1$ and $0$.
\end{exmpl}

It is clear from the description that $\Gamma'$ is isomorphic to the graph $\ExtGraph$ from Section~\ref{sec:complex}. Therefore $\Complex(\Gamma')$ is isomorphic to $\Complex(\ExtGraph)$. We will show below that they are also isomorphic to $\Complex(\Gamma)$. Note that $\Complex(\Gamma'')$ is not typically homeomorphic to these complexes. Indeed if $\Gamma$ consists of two vertices at least one of which is labeled $\infty$, then $\Complex(\Gamma'')$ is $2$-dimensional while $\Complex(\Gamma')$ and $\Complex(\Gamma')$ are $1$-dimensional. The importance of $\Gamma''$ lies not so much in the complex $\Complex(\Gamma'')$ but rather in the group $\GProd{\Gamma''}$.

To show that $\GProd{\Graph''}$ acts on $\Complex(\Graph)$ and $\Complex(\Graph')$ and that both are equivariantly isomorphic, we define a third complex $\WeirdComplex$, that is a coset complex of $\GProd{\Graph''}$. Recall that $\Complex(\Graph)$ is the coset complex of sets of the form $\gen{s}, s \in V_\fin$ and $\qgen{s},\qgen{s^{-1}}, s \in V_\inf$ where $s \in V$, while $\Complex(\Graph')$ is the coset complex of subgroups of the form $\gen{s},s \in V_\fin$ and $\gen{(s,1)}, \gen{(s,-1)}, s \in V_\inf$. The construction of $\WeirdComplex$ is based on the observation that
\begin{equation}
\label{eq:weirds}
\qgen{s}E = \gen{(s,1)}E \quad \text{and}\quad \qgen{s^{-1}}E = \gen{(s,-1)}E
\end{equation}
in $\GProd{\Graph''}$ for $s \in V_\inf$ (this follows from the formulas $s=(s,1) \cdot (s,0)$ and $s^{-1}=(s,-1)\cdot (s,0)$ for $s \in V_\inf$). We therefore define $\Weirds$ to be the poset of sets $\gen{s}E,s \in V_\fin$ as well as those in \eqref{eq:weirds}. Further, $\WeirdCosets$ is defined to be the poset $\GProd{\Graph''}\Weirds$ of cosets of these sets and $\WeirdComplex$ to be the realization of $\WeirdCosets$.

\begin{prop}
The maps $g\qgen{s} \mapsto \alpha(g\qgen{s})E$ and $g\gen{(s,\pm1)} \mapsto \beta(g\gen{s,\pm1})E$ induce ($\GProd{\Graph}$- respectively $\GProd{\Graph'}$-) equivariant isomorphisms $\Complex(\Graph) \to \WeirdComplex$ respectively $\Complex(\Graph') \to \WeirdComplex$. In particular $\GProd{\Graph''}$ acts on $\Complex(\Graph)$ and $\Complex(\Graph')$ and they are equivariantly isomorphic.
\end{prop}

\begin{proof}
Bijectivity of both maps follows from the semidirect product decompositions \eqref{eq:semidirect_product}. Equivariance is clear by construction. The order is preserved since it is just inclusion.
\end{proof}

\providecommand{\bysame}{\leavevmode\hbox to3em{\hrulefill}\thinspace}
\providecommand{\MR}{\relax\ifhmode\unskip\space\fi MR }
% \MRhref is called by the amsart/book/proc definition of \MR.
\providecommand{\MRhref}[2]{%
  \href{http://www.ams.org/mathscinet-getitem?mr=#1}{#2}
}
\providecommand{\href}[2]{#2}

\end{document}